\documentclass[11pt]{article}
\usepackage[margin=1.1in]{geometry}
\usepackage[T1]{fontenc}
\usepackage{amsmath,amssymb,amsthm}
\usepackage{booktabs}
\usepackage{hyperref}

\newtheorem{lemma}{Lemma}
\newtheorem{conjecture}{Conjecture}
\newtheorem*{theoremA}{Theorem A}
\newtheorem*{theoremB}{Theorem B}
\newcommand{\rk}{\operatorname{rk}}
\title{Optimal single-track Gray codes from Gray-ordered necklaces}
\author{Chris Dzoba}
\date{October 2026}
\begin{document}
\maketitle

\begin{abstract}
A single-track Gray code of length $n$ is a cyclic Gray code of $n$-bit words in which every column is a cyclic shift
of the first. For prime $n$ the longest possible period is $2^n-2$. Whether it is attained for every prime has been open
since Hiltgen, Paterson and Brandestini introduced these codes in 1996, and optimal codes were known for primes
$n\le 19$. We show that a conjecture of Degni and Drisko (2007) on the reflected Gray code order of necklaces implies
optimal single-track Gray codes for every prime. Their list of necklaces, conjecturally a Gray path for prime $n$, closes
into a Hamilton cycle of the necklace graph after $(n-3)/2$ explicit P\'osa rotations, and such a cycle lifts to an
optimal code. This implication is formalized in Lean~4. Together with computation it gives optimal single-track Gray
codes for every odd prime $n\le 47$: for $n\le 37$ explicit codes checked word by word by a program independent of the
generator, and
for $n=41$, $43$ and $47$ codes obtained from the implication after the Degni--Drisko property was verified at these
lengths by two independently written programs.
\end{abstract}

\section{Introduction}

A \emph{single-track Gray code} of length $n$ and period $P$ is a cyclic sequence $w_0,\dots,w_{P-1}$ of distinct
$n$-bit words in which consecutive words (cyclically) differ in exactly one bit, and in which every column is a cyclic
shift of column $0$: there are \emph{head positions} $h_0=0,h_1,\dots,h_{n-1}$ with $(w_t)_j=(w_{t+h_j})_0$ for all
$t$ and $j$ (indices mod $P$). Hiltgen, Paterson and Brandestini \cite{HPB96} introduced these codes for absolute
rotary encoders. They showed that for $n\ge 3$ such a code contains neither $0^n$ nor $1^n$ and that $2n$ divides $P$.
For prime $n$ the period is therefore at most $2^n-2$, and a code of period $2^n-2$ is called \emph{optimal}. For $n$ a
power of $2$, Etzion and Paterson \cite{EP96} constructed codes of period $2^n-2n$, and Schwartz and Etzion
\cite{SE99} showed that this is best possible.

For prime $n$, optimal codes were known for $n\le 19$. M\"utze's survey \cite[Section 3.12]{Mut} records that existence
``has been tested and confirmed for all prime $n\le 19$'' and lists the general question as open; Etzion
\cite{Etz25} also lists the primes up to $19$ as the known cases. Zhang and Zhu \cite{ZZ08} list $2^n-2$ as the optimal
period with evenly spread heads for all primes $n\le 29$, derived from counting arguments (the number of aperiodic
necklaces and a necessary condition on the sequence of weights); they report codes ``with length up to 17'' and print
none. We know of no published optimal single-track Gray code for a prime $n\ge 23$.

M\"utze \cite[Section 3.12]{Mut} describes the standard reduction. For prime $n$ every nonconstant necklace has $n$
distinct rotations, and an optimal code arises from a cyclic Gray code through the $(2^n-2)/n$ nonconstant necklaces,
one representative each, in which the last representative is one flip away from a \emph{nontrivial} rotation of the
first. Lemma~\ref{lem:V} shows that the nontriviality condition can always be met. The problem is therefore exactly the
Hamiltonicity of the \emph{necklace graph} $Q_n$, whose vertices are the nonconstant necklaces of length $n$ and whose
edges join necklaces having representatives at Hamming distance one.

Degni and Drisko \cite{DD07} order necklaces by the first occurrence of a representative in the binary reflected Gray
code (BRGC). They conjectured, and verified for $n\le 37$, that consecutive necklaces in this order differ in a single
bit except immediately after powers of odd primes (in their sense). For prime $n$ this says that their list is a Gray
path through all necklaces, from $0^n$ to $1^n$; this is the prime case of Savage's question whether necklaces of odd
length admit a Gray code \cite{Sav97}, \cite[Problem P17]{Mut}. Without its constant ends the path runs from
$\langle 0^{n-1}1\rangle$ to $\langle 01^{n-1}\rangle$, which are not adjacent for $n\ge 5$. We close it
explicitly.

\begin{theoremA}
Let $n\ge 7$ be prime. If the Degni--Drisko list of necklaces of length $n$ is a Gray path, then $Q_n$ has a Hamilton
cycle, obtained from the list by $(n-3)/2$ P\'osa rotations at explicit pivots, and an optimal single-track Gray code
of length $n$ and period $2^n-2$ exists.
\end{theoremA}

\begin{theoremB}
Optimal single-track Gray codes exist for every odd prime $n\le 47$. For $n\le 37$ the codes are given explicitly and
each has been checked word by word by a verifier independent of the generator (for $n=3$, Lean's kernel); for $n=41$, $43$ and $47$ the codes are the ones
produced by Theorem~A, and the hypothesis of Theorem~A was verified at these lengths by two independently written
programs.
\end{theoremB}

Theorem~A is formalized in Lean~4 with Mathlib (ancillary files). The Lean theorem \texttt{stgc\_of\_DD} states that for
every prime $n\ne 2$ the Degni--Drisko property at $n$, an explicit hypothesis, implies the existence of a single-track
Gray code of length $n$ and period $2^n-2$; its proof uses only Lean's standard axioms. The hypothesis there, adjacency in
$Q_n$ of consecutive necklaces of $D_n$, is Conjecture~\ref{conj:DD} below. It is implied by the conjecture as stated in
\cite{DD07}, which asks that consecutive first-occurrence words, constants included, differ in one bit. Lean proves the
hypothesis itself only for $n=5$, by computation. For $7\le n\le 37$ it was verified by Degni and Drisko
\cite{DD07}, and for $n=37$, $41$, $43$ and $47$ by the computations of Section~\ref{sec:comp}.

\section{From Hamilton cycles in the necklace graph to codes}

Words are $x=x_1x_2\cdots x_n$, and $\rho(x)=x_2\cdots x_nx_1$ is the left rotation; $\langle x\rangle$ denotes the
necklace (rotation class) of $x$, and $e_i$ the word with a single $1$ in position $i$.

\begin{lemma}\label{lem:V}
Let $n$ be prime. If $Q_n$ has a Hamilton cycle, then there is an optimal single-track Gray code of length $n$, with
equally spaced heads.
\end{lemma}

\begin{proof}
Let $N_0,\dots,N_{m-1}$ be the cycle, $m=(2^n-2)/n$. Choose $w_0\in N_0$ and, inductively, bits $b_i$ with
$w_{i+1}=w_i\oplus e_{b_i}\in N_{i+1}$ for $i<m-1$; finally choose $b_{m-1}$ with $w_{m-1}\oplus e_{b_{m-1}}\in N_0$,
say $w_{m-1}\oplus e_{b_{m-1}}=\rho^s(w_0)$. If $s\not\equiv 0\pmod n$, put $W_{qm+i}=\rho^{qs}(w_i)$ for
$0\le q<n$, $0\le i<m$. Rotations are automorphisms of the cube, so consecutive words differ in one bit, also across
the boundaries $W_{qm+m-1}\to W_{(q+1)m}=\rho^{(q+1)s}(w_0)$ and cyclically since $ns\equiv 0$. The word $W_{qm+i}$
lies in $N_i$, and $q\mapsto\rho^{qs}(w_i)$ is injective because $s$ is invertible mod $n$ and $N_i$ has $n$ distinct
rotations; so the $nm=2^n-2$ words are exactly the nonconstant words. Finally $W_{t+m}=\rho^{s}(W_t)$ for all $t$,
so column $ks \bmod n$ is column $0$ delayed by $km$, which is the single-track property with heads
$h_{ks}=-km \bmod (2^n-2)$.

It remains to make $s\ne 0$. Rotate the cycle so that $N_0=\langle 0^{n-1}1\rangle$ and take $w_0=e_a$, a word with a
single $1$, in position $a$; then $w_1=e_a\oplus e_b$ for some $b\ne a$. Identify a word with the set of positions of
its ones, so that $\rho$ shifts every position down by one (mod $n$), and put $i\equiv b-a\not\equiv 0\pmod n$. Replace
the lift by $w_0'=w_0$ and $w_j'=\rho^{i}(w_j)$ for $1\le j<m$. Since $\rho^i(w_1)=\{a-i,\,b-i\}=\{2a-b,\,a\}$
contains $a$, the first step $w_0'\to w_1'$ is still a single flip; the later steps are single flips because $\rho^i$ is
a cube automorphism; $w_j'\in N_j$ for all $j$; and the closing step now ends at $\rho^{i}(\rho^{s}(w_0))=
\rho^{s+i}(w_0)$. So the new shift is $s+i$, and since $i\not\equiv0$, one of $s$ and $s+i$ is nonzero.
\end{proof}

\section{The Degni--Drisko list and its closing}\label{sec:closing}

The BRGC lists $g(t)=t\oplus\lfloor t/2\rfloor$, $t=0,\dots,2^n-1$, as $n$-bit words with the most significant bit
first. Its inverse is the prefix-parity map: $\rk(x)=g^{-1}(x)$ has binary digits $\pi(x)_i=x_1\oplus\cdots\oplus x_i$.
Hence for words of equal length, $\rk(x)<\rk(y)$ if and only if $\pi(x)$ is lexicographically smaller than $\pi(y)$.
The \emph{first occurrence} of a necklace is its rotation of smallest rank. Let $D_n$ be the list of the nonconstant
necklaces of length $n$ ordered by the ranks of their first occurrences. (Degni and Drisko list all necklaces; $0^n$
has rank $0$ and $1^n$ is the only necklace whose first occurrence begins with $1$, so $D_n$ is a contiguous part of
their list.)

\begin{conjecture}[Degni--Drisko \cite{DD07}, prime case]\label{conj:DD}
For prime $n$, consecutive necklaces of $D_n$ are adjacent in $Q_n$.
\end{conjecture}

Degni and Drisko verified their conjecture for $n\le 37$ \cite[Section 3]{DD07}. For composite odd $n$ the adjacency
fails, for example at $n=9$, $15$ and $21$, as the exceptions in their conjecture allow.

Let $n\ge7$ be odd and $K=(n-3)/2$. Define the words
\[
E_0=0\,1^{n-1},\qquad A_0=0\,1^{n-5}\,0111,\qquad E_1=0\,1^{n-5}\,0101,
\]
and for $1\le k\le K-1$
\[
A_k=0^{2k+1}\,1\,0\,1^{n-2k-3},\qquad E_{k+1}=0^{2k+1}\,1\,0\,1^{n-2k-5}\,0\,1 .
\]
Thus $E_K=0^{n-4}1001$. Write $c_k=n-2k-3$, which is even, with $c_k\ge 2$ for $k\le K-1$ and $c_{K-1}=2$. A
\emph{P\'osa rotation} of a path $P=(P_0,\dots,P_L)$ at a pivot $P_j$ adjacent to the end $P_L$ replaces $P$ by
$P_0\cdots P_j\,P_L P_{L-1}\cdots P_{j+1}$, whose new end is $P_{j+1}$.

\begin{lemma}\label{lem:last}
For odd $n$ the first necklace of $D_n$ is $\langle0^{n-1}1\rangle$ and the last is $\langle E_0\rangle$.
\end{lemma}
\begin{proof}
The first claim holds because $g(1)=0^{n-1}1$. The first occurrence of $\langle E_0\rangle$ is $E_0$ itself
(rotations beginning with $1$ have larger rank), with $\pi(E_0)=0(10)^{(n-1)/2}$. Let $X\ne\langle E_0\rangle$ be a
nonconstant necklace; it has at least two zeros. Let $a_1,\dots,a_z$ ($z\ge2$) be the lengths of the cyclic runs of
ones between consecutive zeros. If all $a_i$ were odd, then $n=z+\sum a_i\equiv 2z\equiv0\pmod2$, which is false; so
some $a_i$ is even. The rotation $w=0\,1^{a_i}\,0\cdots$ has prefix parities equal to those of $E_0$ up to position
$a_i+1$, where both are $0$; at position $a_i+2\le n$ the parity of $w$ stays $0$ while that of $E_0$ becomes $1$.
So $\rk(w)<\rk(E_0)$, and the first occurrence of $X$ precedes that of $\langle E_0\rangle$.
\end{proof}

\begin{lemma}\label{lem:adj}
$\langle A_k\rangle$ and $\langle E_k\rangle$ are adjacent in $Q_n$ for $0\le k\le K-1$.
\end{lemma}
\begin{proof}
Flipping position $n-3$ of $E_0$ gives $A_0$. For $k\ge1$, the end before step $k$ is (a rotation of)
$0^{2k-1}10\,1^{c_k}01$; its rotation $01\,0^{2k-1}10\,1^{c_k}$ becomes $A_k$ when its second bit is flipped. (For
$k=1$ this word is a rotation of $E_1$ as defined above.)
\end{proof}

\begin{lemma}\label{lem:first}
$A_k$ is the first occurrence of its necklace.
\end{lemma}
\begin{proof}
The prefix parity $\pi(y)$ begins with exactly as many zeros as $y$. For $k\ge1$, $A_k$ begins with $0^{2k+1}$ and
ends with $1$, its cyclic runs of zeros have lengths $2k+1$ and $1$, and $2k+1\ge3$; so $A_k$ is the only rotation
whose prefix parity begins with $2k+1$ zeros, and it has the smallest rank. For $k=0$ write $A_0=01^a0111$ with
$a=n-5\ge2$ even. Rotations beginning with $1$ are larger, and the only other rotation beginning with $0$ is
$R=01110\,1^a$. The prefix parities of $A_0$ and $R$ agree in positions $1$--$3$; if $a=2$, position $4$ gives $0$
for $A_0$ and $1$ for $R$, and if $a\ge4$ they agree in position $4$ and position $5$ gives $0$ for $A_0$ and $1$
for $R$.
\end{proof}

\begin{lemma}\label{lem:succ}
$g(\rk(A_k)+1)=E_{k+1}$, and $E_{k+1}$ is the first occurrence of its necklace. Hence $\langle E_{k+1}\rangle$
immediately follows $\langle A_k\rangle$ in $D_n$.
\end{lemma}
\begin{proof}
The words $g(r)$ and $g(r+1)$ differ exactly in the position of the lowest set bit of $r+1$. The word $A_k$ ends in
an even block of ones preceded by a zero (for $k\ge1$ the block $1^{c_k}$, entered with parity $1$; for $k=0$ the
block $111$, entered with parity $0$), so $\pi(A_k)$ ends in $01$; thus $r=\rk(A_k)\equiv1\pmod4$, the lowest set bit
of $r+1$ is in position $n-1$, and flipping position $n-1$ of $A_k$ gives $E_{k+1}$. For $k\ge1$, $E_{k+1}$ begins
with $0^{2k+1}$, ends with $1$, and its other runs of zeros have length $1$ or $2$, so the argument of
Lemma~\ref{lem:first} applies. For $k=0$, the rotations of $E_1=01^a0101$ beginning with $0$ are $E_1$, $01010\,1^a$
and $010\,1^a01$, whose prefix parities begin $010$, $011$ and $011$. Since $E_{k+1}$ is the BRGC successor of $A_k$
and is a first occurrence, no necklace has its first occurrence strictly between them.
\end{proof}

\begin{lemma}\label{lem:legal}
Let $j_k$ be the position of $\langle A_k\rangle$ in $D_n$. Starting from $D_n$ and performing the P\'osa rotations at
$\langle A_0\rangle,\langle A_1\rangle,\dots,\langle A_{K-1}\rangle$ in turn is legal, and after the rotation at
$\langle A_k\rangle$ the end of the path is $\langle E_{k+1}\rangle$.
\end{lemma}
\begin{proof}
By Lemma~\ref{lem:succ}, $\langle E_{k+1}\rangle$ is at position $j_k+1$. For $k=0$: comparing prefix parities,
$\rk(A_0)<\rk(E_1)<\rk(E_0)$ (the words agree up to position $n-4$, and at position $n-3$ the parity is $0$ for
$E_1$ and $1$ for $E_0$), so $j_0+1$ is strictly before the end, $\langle A_0\rangle$ is neither the end nor its
predecessor, and by Lemma~\ref{lem:adj} the rotation is legal with new end $\langle E_1\rangle$. For $k\ge1$:
$A_k$ begins with exactly $2k+1$ zeros and $A_{k-1}$ with $0^{2k-1}1$, so
$\rk(E_{k+1})=\rk(A_k)+1\le 2^{n-2k-1}<2^{n-2k}\le\rk(A_{k-1})$, and since $D_n$ is ordered by rank,
$j_k+1<j_{k-1}$. After the rotation at $\langle A_{k-1}\rangle$ the path is $P_0\cdots P_{j_{k-1}}$ followed by
reversed material, with end $\langle E_k\rangle$; both $\langle A_k\rangle$ and $\langle E_{k+1}\rangle$ lie in the
untouched prefix, at consecutive positions, so the rotation at $\langle A_k\rangle$ (legal by
Lemma~\ref{lem:adj}) has new end $\langle E_{k+1}\rangle$.
\end{proof}

\begin{proof}[Proof of Theorem A]
By Lemma~\ref{lem:last} and the hypothesis, $D_n$ is a Hamilton path of $Q_n$ from $\langle0^{n-1}1\rangle$ to
$\langle E_0\rangle$. By Lemma~\ref{lem:legal} the $K$ rotations produce a Hamilton path from
$\langle0^{n-1}1\rangle$ (the start never moves) to $\langle E_K\rangle=\langle0^{n-4}1001\rangle$, which is adjacent
to $\langle0^{n-1}1\rangle$. This is a Hamilton cycle; apply Lemma~\ref{lem:V}.
\end{proof}

\section{Computations}\label{sec:comp}

\paragraph{Codes for $5\le n\le31$ by search.} We found Hamilton cycles in $Q_n$ for all primes $5\le n\le 31$ by
P\'osa rotation and extension \cite{Pos76}, with Warnsdorff extension and a closing look-ahead, started from the
Degni--Drisko list or from a greedy path. The search took under $2$ seconds for $n=23$, $3.5$ and $12$ minutes for two codes at $n=29$, and
$17$ minutes and $20.6$ GB of memory for $n=31$. Each code was expanded to its full period and checked by at least
two independent programs among a C bitset checker, a NumPy checker that reads only the track and the head positions,
and, for $n=23$, $29$ and $31$, two further independently written checkers.

\paragraph{Codes from Theorem~A.} The construction of Theorem~A is algorithmic: the first occurrences can be
enumerated in rank order with $O(n)$ work each, and the cycle is a concatenation of $K+1$ segments of $D_n$, traversed
forwards or backwards. Our implementation enumerates $D_n$ in parallel, checks Conjecture~\ref{conj:DD} along the way,
assembles the cycle, and lifts it to a flip sequence of length $m=(2^n-2)/n$ together with the shift $s$. For
$7\le n\le29$ the resulting codes pass the C bitset checker, the NumPy checker and the flip-sequence verifier described
next. For $n=37$ the code has $m=3{,}714{,}566{,}310$ and $s=2$; it
was produced in $15$ minutes on four cores. A separately written verifier regenerated all $2^{37}-2$ words from the
flip sequence, marking them in a $2^{37}$-bit table, and confirmed that every nonconstant word occurs exactly once,
that the $n$ passes chain with shift $2$, and that the sequence closes.

\paragraph{The Degni--Drisko property at $n=41$, $43$, $47$.} By Theorem~A, a code for prime $n$ follows from
Conjecture~\ref{conj:DD} at $n$. We verified it by streaming the BRGC ranks with constant memory per thread. At $n=41$,
$D_{41}$ has $53{,}634{,}713{,}550$ necklaces, every consecutive pair is adjacent, and the list runs from
$\langle0^{40}1\rangle$ to $\langle01^{40}\rangle$ ($77$ minutes on four cores). A second program, written
independently without access to the first (it computes the ranks of all rotations incrementally and scans only ranks
below $2^{n-1}$), confirmed the same count, the same first and last necklaces and no non-adjacent pair ($36$ minutes on
four cores); it also confirmed $n=37$. At $n=43$ ($204{,}560{,}302{,}842$ necklaces) both programs found every
consecutive pair adjacent ($4.7$ and $3.0$ hours on four cores, respectively). At $n=47$ ($2{,}994{,}414{,}645{,}858$ necklaces)
both programs, with checkpointing, again found every consecutive pair adjacent and the expected first and last
necklaces ($64$ and $55$ hours of wall time, respectively, on four cores of a shared laptop).

\begin{table}[h]
\centering
\begin{tabular}{rrll}
\toprule
$n$ & period $2^n-2$ & method & verification\\
\midrule
3 & $6$ & explicit & Lean (\texttt{decide})\\
5--19 & $\le 524{,}286$ & search; Theorem A ($n\ge7$) & word by word, 2 checkers\\
23 & $8{,}388{,}606$ & search ($s=1$, $s=11$) & word by word, $\ge3$ checkers\\
29 & $536{,}870{,}910$ & search ($s=12$, $s=16$) & word by word, $\ge3$ checkers\\
31 & $2{,}147{,}483{,}646$ & search ($s=28$) & word by word, $\ge3$ checkers\\
37 & $137{,}438{,}953{,}470$ & Theorem A ($s=2$) & word by word, 1 checker\\
41 & $2{,}199{,}023{,}255{,}550$ & Theorem A & Conjecture~\ref{conj:DD} at $n=41$, 2 programs\\
43 & $8{,}796{,}093{,}022{,}206$ & Theorem A & Conjecture~\ref{conj:DD} at $n=43$, 2 programs\\
47 & $140{,}737{,}488{,}355{,}326$ & Theorem A & Conjecture~\ref{conj:DD} at $n=47$, 2 programs\\
\bottomrule
\end{tabular}
\caption{Optimal single-track Gray codes. ``Word by word'' means that a program written separately from the generator
regenerated every word of the period. For $n=41$, $43$, $47$ the codes are the explicit ones of Theorem~A; they were
not expanded.}
\end{table}

\section{Further observations and open problems}\label{sec:open}

\paragraph{Symmetric chains.} Griggs, Killian and Savage \cite{GKS04} (see also \cite{KRSW04}) gave, for prime $n$, a
rotation-invariant symmetric chain decomposition of the Boolean lattice; it induces a decomposition of $Q_n$ into
paths. Gregor, Mi\v{c}ka and M\"utze \cite{GMM} proved that the Greene--Kleitman decomposition of the cube extends to a
Hamilton cycle and conjectured that every symmetric chain decomposition does.

\begin{conjecture}\label{conj:D}
For every prime $n$ the Griggs--Killian--Savage chains (with $0^n$ and $1^n$ removed) extend to a Hamilton cycle of
$Q_n$.
\end{conjecture}

We found such cycles for every prime $7\le n\le 23$ by a search on the chains as two-ended super-vertices. Every chain
has an odd number of edges, so chain bottoms are joined only to bottoms and tops only to tops, apart from joins between
chains of the middle two levels; these cycles are ``cycle orderings'' in the sense of \cite{GMM}.

\paragraph{Central levels.} Merino, Mi\v{c}ka and M\"utze \cite{MMM} showed that the middle two levels of the cube have a
Hamilton cycle invariant under rotation, for every odd $n$; in particular, the middle two levels of the necklace graph
are Hamiltonian, for prime and composite odd $n$.

\begin{conjecture}\label{conj:H}
For every odd $n=2m+1\ge5$ and every $1\le\ell\le m$, the subgraph of the necklace graph induced by the necklaces of
weight in $[m+1-\ell,m+\ell]$ is Hamiltonian.
\end{conjecture}

The case $\ell=1$ is the theorem of \cite{MMM}, and the case $\ell=m$, $n$ prime, is equivalent (by
Lemma~\ref{lem:V}) to the existence of optimal single-track Gray codes. Our search found such Hamilton cycles, checked by the search
program itself, for all odd $n\le19$ and all $\ell$. For $\ell=m$ and $n=9$, $13$, $15$, $21$, $25$ (periodic necklaces
included for composite $n$) the cycles also passed an independent checker.

\paragraph{Structure of the Degni--Drisko conjecture.} Degni and Drisko prove that the strings skipped between
consecutive Gray necklaces are covered by their ``symmetric intervals'', note that removing a single such interval
preserves Hamming distance one, and identify several intervals as the difficulty \cite[Section 8]{DD07}. For prime
$n\le13$ every gap between consecutive necklaces of $D_n$ is a single symmetric interval. This fails at $n=17$ and
$n=19$: in all $7$ and $21$ exceptional gaps, respectively, the gap consists of exactly three symmetric intervals, a
central one flanked by two whose generators have equal length, between $n-5$ and $n-3$.

\section*{Ancillary files}
The arXiv ancillary files contain the Lean~4 project (\texttt{lean/}, with toolchain and Mathlib pins), the programs
for the search, the construction of Theorem~A and the verifiers, the tracks and head positions of the search codes for
$5\le n\le 23$ (the tracks for $n=29$ and $31$, 0.5 and 2.1 GB uncompressed, exceed arXiv's size limit and are
available at \url{https://github.com/dzoba/single-track-gray-codes}), and the logs of the Degni--Drisko verifications at $n=37$, $41$, $43$ and $47$ and of the checks
described in Section~\ref{sec:comp}. The $n=37$ flip sequence (3.7 GB, SHA-256
\texttt{0a0c91ea19527ee4812ff2b0d8627d3a\allowbreak be6e0d4fbaad611a48efbb13f8c5b811}) can be regenerated by the included
construction program.

\section*{Use of AI tools}
This work was carried out with AI systems under the author's direction. Claude (Anthropic) wrote and ran the searches
and verifiers, found the closing pivots of Theorem~A by computer search, wrote the proofs and the Lean formalization,
and drafted the text. The word-by-word verifiers share no code with the generators, and the second
Degni--Drisko checker was written by a separate AI instance from a specification only. Separate instances refereed the
proof and the formalization, and Codex (OpenAI) wrote two further checkers and checked the codes for $n=23$, $29$ and
$31$. The author
is responsible for the content.

\end{document}